\documentclass[a4paper,12pt]{article}
\usepackage[T1]{fontenc}
\usepackage[latin2]{inputenc}
\usepackage[a4paper,left=1.8cm,right=1.8cm,top=2cm,bottom=2cm]{geometry}
\usepackage {times,authblk,url} 
\usepackage {amsmath}
\usepackage {amsfonts}
\usepackage {amssymb}
\usepackage {graphicx}
  
\newtheorem{thm}{\bfseries Theorem}
\newtheorem{cor}[thm]{\bfseries Corollary}     

\newtheorem{cl}[thm]{\bfseries Claim}

\newenvironment{proof}{\medskip \noindent{\sc Proof:}}{\quad$\Box$\par\medskip} 
\DeclareMathOperator{\opt}{opt}
\DeclareMathOperator{\supp}{supp}
\DeclareMathOperator{\diam}{diam}

\title{Optimal pebbling and rubbling of graphs with given diameter }
\author[1,2]{Ervin Gy\H{o}ri\thanks{gyori.ervin@renyi.mta.hu}}
\author[3,4]{Gyula Y. Katona\thanks{kiskat@cs.bme.hu}}
\author[3]{L\'aszl\'o F. Papp\thanks{lazsa@cs.bme.hu}}
\affil[1]{Alfr\'ed R\'enyi Institute of Mathematics, Budapest, Hungary}
\affil[2]{Department of Mathematics, Central European University, Budapest, Hungary}
\affil[3]{Department of Computer Science and
Information Theory, Budapest University of Technology and Economics, Hungary}
\affil[4]{MTA-ELTE Numerical Analysis and Large Networks Research Group, Hungary}

\begin{document}

\maketitle

\begin{abstract}
 A pebbling move on a graph removes two pebbles from a vertex and adds
  one pebble to an adjacent vertex. A vertex is reachable
  from a pebble distribution if it is possible to move a pebble to
  that vertex using pebbling moves. The optimal pebbling number $\pi_{\opt}$ is
  the smallest number $m$ needed to guarantee a pebble distribution of
  $m$ pebbles from which any vertex is reachable. A rubbling move is similar to 
  a pebbling move, but it can remove the two pebbles from two different vertex.
  The optimal rubbling number $\rho_{\opt}$ is defined analogously to the optimal pebbling number. 
  In this paper we give lower bounds
  on both the optimal pebbling and rubbling numbers by the distance $k$ domination number. 
  With this bound we prove that for each $k$ there is a graph $G$ with diameter $k$ such that 
  $\rho_{\opt}(G)=\pi_{\opt}(G)=2^k$.


\end{abstract}

\begin{quote}
\end{quote}
\vspace{5mm}

\section{Introduction} 

Graph pebbling is a game on graphs initialized by a question of Saks and Lagarias, which was answered by Chung in 1989 \cite{chung}. Its roots are originated in number theory. 

Each graph in this paper is simple and connected. We denote the vertex set and the edge set of graph $G$ with $V(G)$ and $E(G)$, respectively. The distance between vertices $u$ and $v$, denoted by $d(u,v)$, is the minimum number of edges contained in the shortest path connecting $u$ and $v$. We use $\diam(G)$ for the diameter of $G$. 

We write $G\square H$ for the Cartesian product of graphs $G$ and $H$. The vertex set of graph ${G\square H}$ is $V(G)\times V(H)$ and vertices $(g,h)$ and $(g',h')$ are adjacent if and only if either $g=g'$ and ${\{h,h'\}\in E(H)}$, or $h=h'$ and $\{g,g'\}\in E(G)$. 
We use $G^{\square d}$ as an abbreviation of $G\square G \square \dots \square G$ where $G$ appears exactly $d$ times.  
 
 A \emph{pebble distribution} $P$ on graph $G$ is a function mapping the vertex set to nonnegative integers. We can imagine that each vertex $v$ has $P(v)$  pebbles. A \emph{pebbling move} removes two pebbles from a vertex and places one to an adjacent one. A pebbling move is \emph{allowed} if and only if the vertex loosing pebbles has at least two pebbles. 

A sequence of pebbling moves is called \emph{executable} if for any $i$ the $i$th move is allowed under the the distribution obtained by the application of the first $i-1$ move. The pebble distribution which we get from $P$ after the execution of the sequence of pebbling moves $\sigma$ is denoted by $P_{\sigma}$.
 
 A vertex $v$ is \emph{reachable} under a distribution $P$, if there is an executable sequence of pebbling moves $\sigma$, such that $P_{\sigma}(v)\geq 1$. We say that a distribution $P$ is \emph{solvable} if each vertex is reachable under $P$. 
The size of a pebble distribution $P$ is $\sum_{v\in V(G)} P(v)$ which we denote by $|P|$.
A pebble distribution $P$ on a graph $G$ will be called {\it optimal} if it is solvable and 
its size is the smallest possible. The size of an optimal pebble distribution is called \emph{the optimal pebbling} number and denoted by $\pi_{opt}(G)$.

In \cite{rubbling} the authors invented a version of pebbling called rubbling. The only difference between the definitions of pebbling and rubbling is that there is an additional available move. A \emph{strict rubbling move} removes two pebbles in total but it takes them from two different vertices then it places one pebble at one of their common neighbours. Thus a strict rubbling move is allowed if it removes pebbles from vertices who share a neighbour and both of them has a pebble. 
A rubbling move is either a pebbling move or a strict rubbling move. 
If we replace pebbling moves with rubbling moves everywhere in the definition of the optimal pebbling number, then we obtain the \emph{optimal rubbling number}, which is denoted by $\rho_{\opt}$. 

There are not  many results on rubbling, only two articles \cite{rub1,rub2} appeared about rubbling so far. On the other hand, the optimal pebbling number of several graph families are known. For example exact values were given for paths and cycles \cite{path1,path2}, ladders \cite{ladder}, caterpillars \cite{caterpillar}, m-ary trees \cite{m-ary} and staircase graphs \cite{stairs}. 
However, determining the optimal pebbling number for a given graph is NP-hard \cite{NPhard}. There are also some known bounds on the optimal pebbling number. One of the earliest is that
$\pi_{\opt}(G)\leq 2^{\diam(G)}$.

Placing $2^{\diam(G)}$ pebbles to a single vertex always creates a solvable distribution, but usually much less pebbles are enough to construct a solvable distribution. It is a natural question, if there are graphs with arbitrary large diameter where this amount of pebbles is required for an optimal pebbling?

The answer is positive and it was given in \cite{smalldiam}. However, the proof  in \cite{smalldiam} is incorrect. The authors gave a set of graphs and claimed that they have this property, but we will show during the proof of Claim \ref{kevesebb} that it is not true. 

Herscovici \emph{et al.}~in \cite{product} proved that $\pi_{\opt}(K_m^{\square d})=2^d$ if $m>2^{d-1}$. In fact, a more general statement is proved in \cite{product}, but this is enough for our purposes. The diameter of these graphs is $d$, therefore they  prove the sharpness of the diameter bound. 

We can ask, what happens when we consider rubbling instead of pebbling? Unfortunately the proof of Herscovici \emph{et al.}~rely on several phenomena true for pebbling but false for rubbling. We answer this question and prove that $\rho_{\opt}(K_m^{\square d})=2^d$ if $m\geq 2^{d}$. Since $\rho_{\opt}(G)\leq \pi_{\opt}(G)$, it is also a new short proof for the pebbling case. Our method uses the concept of distance domination.   

A distance $k$ domination set $S$ of a graph is a subset of the vertex set such that for each vertex $v$ there is an element $s$ of $S$ whose distance from $v$ is at most $k$. The distance $k$ domination number of a graph, denoted by $\gamma_k$, is the size of the smallest distance $k$ domination set. 

First we prove that $\rho_{\opt}(G)\geq \min\left(\gamma_{k-1}(G),2^k\right)$ for each $k$, 
then we give an improved lower bound using both $\gamma_{k-1}$ and $\gamma_{k-2}$. 

Finally we use these bounds to show that $\pi_{\opt}(K_3\square K_3\square K_5)=6$.

\section{Counterexample to the proof of Muntz \textit{et al.}}

Muntz \textit{et al.}~give an iterative construction of graphs. They claim in \cite{smalldiam}, that if $G$ is a graph 
with diameter $d$ and its optimal pebbling number is $2^d$, then $G\square K_{2^d+1}$ is a graph with diameter 
$d+1$ and optimal pebbling number $2^{d+1}$. It is easy to see that $\diam(G\square K_{2^d+1})=d+1$, however
its optimal pebbling number is not necessarily $2\pi_{\opt}(G)$. 

Muntz \textit{et al.}~choose $K_3$ as a starting graph. The third graph in the sequence is $K_3\square K_3\square K_5$.
The optimal pebbling number of this graph is not $8$, as the authors claimed. 

\begin{cl}
The optimal pebbling number of $K_3\square K_3\square K_5$ is at most $6$.
\label{kevesebb}
\end{cl} 

\begin{proof}
A solvable distribution with $6$ pebbles is given in Figure \ref{ellenpelda}. We can move two pebbles to each vertex of the leftmost 
$K_3\square K_3$. Since each other vertex is connected to these vertices, all vertices are reachable.
\end{proof}

\begin{figure}
\centering
\includegraphics[scale=1]{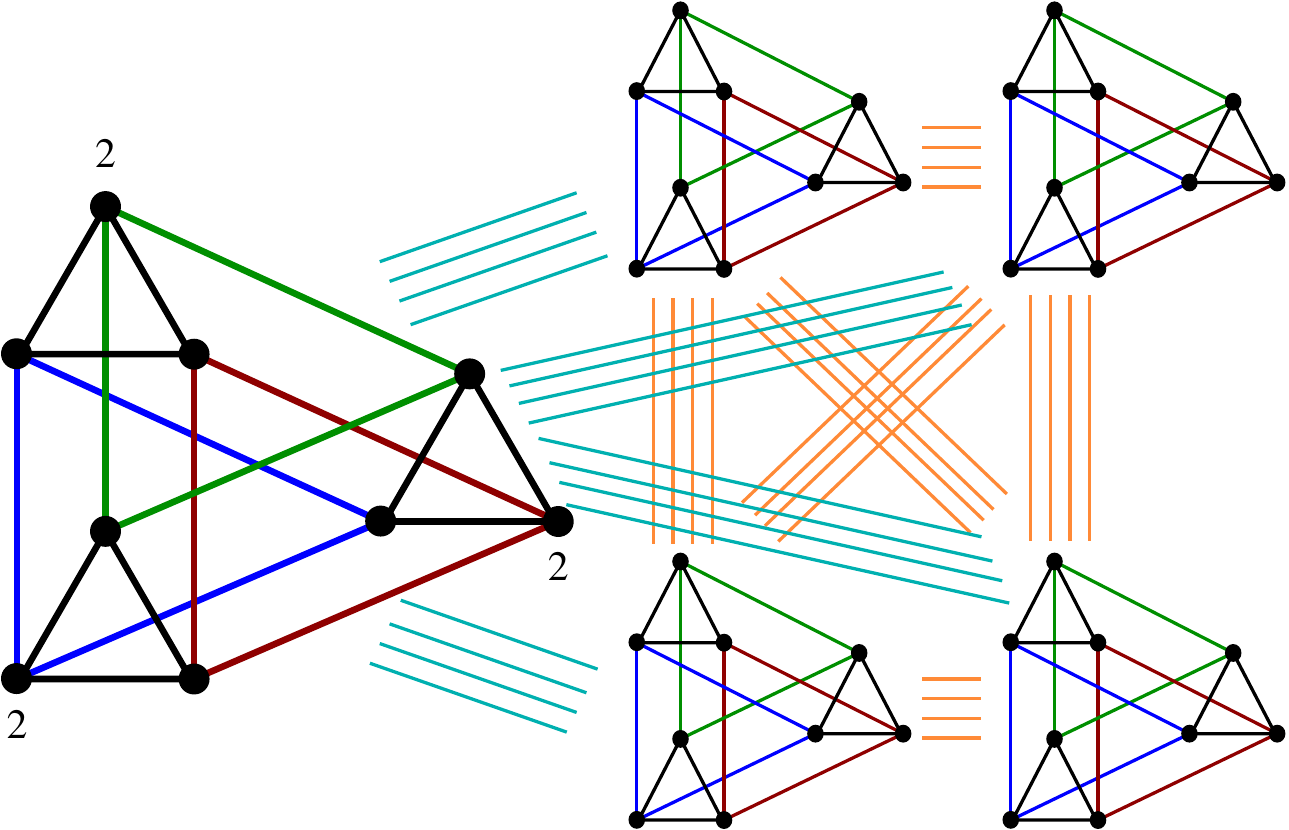}
\caption{An optimal distribution of $K_3\square K_3\square K_5$ using $6$ pebbles.}
\label{ellenpelda}
\end{figure}

Furthermore, 
all later graphs in the sequence are counterexamples. Because if we take a solvable distribution of $G$ and use the double of its pebbles on a copy of $G$ in 
$G\square K_n$, then we get a solvable distribution of $G\square K_n$. Thus if $\pi_{opt}(G)<2^d$, then $\pi_{opt}(G\square K_n)<2^{d+1}$.

Besides, it can be proven that changing the starting graph does not help, the construction fails. 

\section{A lower bound given by the distance domination number}
We establish our first lower bound on the optimal pebbling and rubbling numbers using the distance $k$ domination number. 

\begin{thm}
Let $G$ be a connected graph and $k$ be an integer greater than one
, then: 
$$\rho_{\opt}(G)\geq \min\left(\gamma_{k-1}(G),2^k\right).$$
\label{kdomlemma}
\end{thm}

The \emph{support} of a pebble distribution $P$, denoted by $\supp(P)$ is the set of vertices containing at least one pebble.

The \emph{weight function} of $P$, which is defined on the vertex set of $G$, is 
$W_P(u)=\sum_{v\in V(G)}P(v)2^{-d(u,v)}$.

Clearly, if a vertex is reachable under $P$, then its weight is at least one. This is true for both pebbling and rubbling.

\begin{proof}
Consider a pebble distribution $P$ whose size is less than both $\gamma_{k-1}(G)$ and $2^k$. Hence $\supp(P)$ is not a distance $k-1$ dominating set. There is a vertex $v$ whose distance from $\supp(P)$ is $k$. Therefore the weight of this vertex is at most $\frac{1}{2^k}|P|<1$, hence $v$ is not reachable under $P$. So a solvable pebble distribution has at least $\min(\gamma_{k-1}(G),2^k)$ pebbles. 
\end{proof}

We are free to choose $k$. The best bound is obtained when $\gamma_{k-1}\approx 2^k$. 

Notice that the proof exploited that each vertex contains integer pebbles and that the degradation of pebbles is exponential in sense of the distance. On the other hand, we have not used that a pebbling or a rubbling move removes integer number of pebbles. Therefore this method also works when a pebble can be broken to arbitrary small pieces. Hence it also gives a bound on the optimal integer fractional covering ratio which is defined in \cite{gridnote}.

\section{The optimal rubbling number of $K_m^{\square d}$ is $2^d$ if $m\geq 2^d$}

Let $\Sigma_{m,k}$ be the following graph: We choose an alphabet $\Sigma$ of size $m$. The vertices of $\Sigma_{m,k}$ are the words over $\Sigma$ of length $k$. Two vertices are adjacent if and only if the corresponding words differ only at one position, roughly speaking their Hamming distance is one. It is well known that $\Sigma_{m,k}\simeq K_m^{\square k}$. We use this coding theory approach because it is more natural for us to interpret the following proofs in this language. 

It is easy to see that  $\diam(\Sigma_{m,k})=k$: We have to change all $k$ characters of word $a,a,\dots,a$ to obtain $b,b,\dots,b$, each of the changes requires passing through an edge. We can obtain any word from any other by changing each character at most once, hence $\diam(\Sigma_{m,k})=k$. 

\begin{cl}
$\gamma_{k-1}(\Sigma_{m,k})=m$.
\end{cl}

\begin{proof}
The set containing all constant words over alphabet $\Sigma$ with length $k$  is a distance $k-1$ dominating set, because it is enough to change at most $k-1$ characters of a $k$ long word to obtain a constant one. The number of these words is $m$.

Let $A$ be a set of words over alphabet $\Sigma$ with length $k$ such that the size of $A$ is $m-1$. Consider the $i$th characters of all words contained in $A$. The pigeonhole principle implies that there is a character $c_i\in \Sigma$ which does not appear among them. Such a character exists for each position. Consider the word $c=c_1c_2\dots c_k$. We have to change all of its characters to obtain a word contained in $A$, thus its distance from $A$ is $k$ so $A$ is not a distance $k-1$ dominating set.
 \end{proof} 

\begin{thm}
Both the optimal pebbling and optimal rubbling number of $K_m^{\square d}$ is $2^d$ if $m\geq 2^d$.
\end{thm}

\begin{proof}
We have already seen that $2^d$ pebbles at a single vertex is enough to construct a solvable pebble distribution even if we consider only pebbling moves. 

For the lower bound we set $k$ as $d$ and apply Theorem \ref{kdomlemma}. The obtained lower bound is also $2^d$.
\end{proof}

\section{Lower bounds using both $\gamma_{k-1}$ and $\gamma_{k-2}$}

To improve Theorem \ref{kdomlemma}, we have to use several properties of pebbling and rubbling. Therefore the obtained bounds are no longer the same for $\pi_{\opt}$ and $\rho_{\opt}$.

Let $S$ be a subset of $V(G)$. The open neighbourhood of vertex $v$ is the set of vertices which are adjacent to $v$. The closed neighbourhood contains the adjacent vertices plus the vertex itself. The closed neighbourhood of set $S$, denoted by $N[S]$, is defined as the union of the closed neighbourhoods of vertices contained in $S$. We write $N(S)$ for the open neighbourhood of $S$ which is defined as $N[S]\setminus S$.

Let $\sigma$ be a sequence of pebbling moves and let $M$ be a pebbling move which contained in $\sigma$. We write $\sigma\setminus M$ for the sequence of pebbling moves which we get after we delete the last appearance of $M$. If we add an additional pebbling move $T$ to the beginning of $\sigma$, then we denote the obtained sequence by $T\sigma$.

Let $P$ be a pebble distribution on $G$ and $S$ be an arbitrary subset of $V(G)$. Then the the \emph{restriction} of $P$ to $S$ is a pebble distribution which is defined as follows:
$$P|_{S}=\begin{cases}
 P(v) &\text{if }v\in S\\
 0 &\text{otherwise}
\end{cases}$$

\begin{thm}
For all $k\geq 3$ and any graph $G$ whose edge set is non empty we have:
 
$$\pi_{\opt}(G)\geq \min\left(2^k,\gamma_{k-1}(G)+2^{k-2},\gamma_{k-2}(G)+1\right)$$
\label{kozepso}
\end{thm}

\begin{proof}
Consider a solvable pebble distribution $P$. We have already seen that $|P|\geq \min(\gamma_{k-1}(G),2^k)$.

Assume that $|P|< \min(\gamma_{k-2}(G)+1,2^k)$. Either $\supp(P)$ is not a distance $k-2$ domination set or each vertex has at most one pebble. In the later case there are no available pebbling moves but there are vertices which do not have pebbles, so they are not reachable which is a contradiction.  

In the other case, there is a vertex $v$ whose distance from $\supp(P)$ is at least $k-1$. On the other hand, $\supp(P)$ has to be a distance $k-1$ domination set,
since otherwise $2^k$ pebbles would be required to reach some of the vertices.

Let $\sigma$ be an executable sequence of pebbling moves moving a pebble to $v$.
We say that a subdivision of $\sigma$ to two subseqences $\tau$ and $\mu$ is proper if $\tau$ and $\mu$ are executable under $P$ and $P_\tau$, respectively and $\mu$ does not contain a move which removes a pebble from $\supp(P)$.
We chose a proper subdivision where the size of $\mu$ is maximal.

We execute $\tau$ and investigate the obtained distribution $P_{\tau}$. We show that $\supp(P_{\tau})\subseteq N[\supp(P)]$:
Assume that a vertex outside of $ N[\supp(P)]$ has a pebble under $P_\tau$. Then the last pebbling move $M$ which placed it there does not remove pebbles from $\supp(P)$. $\tau\setminus M$ is executable and if we put $M$ to the beginning of $\mu$ then $M\mu$ is also executable under $P_{\tau\setminus M}$. Furthermore $M\mu$ does not remove a pebble from $\supp(P)$, thus $\tau\setminus M$, $M\mu$ is a proper subdivision of $\sigma$ which contradicts with the maximality of $\mu$. Therefore $\supp(P_{\tau})\subseteq N[\supp(P)]$.

At each vertex of $\supp(P)$ the execution of $\tau$ either leaves a pebble or it removes at least two pebbles by a pebbling move which consumes one pebble. Thus at most $|P|-|\supp(P)|$ pebbles arrive at $N(\supp(P))$ after the execution of $\tau$.

$\mu$ uses only these pebbles and moves a pebble to $v$. Therefore $v$ is reachable under $P_{\tau}|_{N(\supp(P))}$. The distance of $v$ from $\supp(P_{\tau})$ is at least $k-2$, therefore
$2^{2-k}(|P|-|\supp(P)|)\geq W_{P\tau |N(\supp(P))}(v)\geq 1$.
Since $|\supp(P)|\geq \gamma_{k-1}(G)$, we get that $|P|\geq 2^{k-2}+\gamma_{k-1}(G)$.

So either $|P|\geq 2^{k-2}+\gamma_{k-1}(G)$ or our assumption was false and $|P|\geq \min(\gamma_{k-2}(G)+1,2^k)$. Altogether these imply the desired result. 
\end{proof}

If we talk about rubbling, then there are two main differences. First, a distribution which places at most one pebble everywhere and leaving a vertex without a pebble can be solvable. Second, a rubbling move can remove pebbles from two vertices and consume just one pebble, hence we can state just that $|P|-\frac{|\supp(P)|}{2}$ pebbles arrive at $N(\supp(P))$ after the execution of $\tau$. If we change the above proof accordingly, then we get the following improved version of Theorem \ref{kdomlemma} for rubbling:

\begin{thm}
For all $k\geq 2$ and all graphs $G$ we have:
 
$$\rho_{\opt}(G)\geq \min\left(2^k,\max\left(\frac{\gamma_{k-1}(G)}{2}+2^{k-2},\gamma_{k-1}(G)\right),\gamma_{k-2}(G)\right).$$
\end{thm}

We can slightly improve the pebbling result if we do some case analysis.

\begin{thm}
For all $k\geq 3$ and any graph $G$ whose edge set is non empty 
we have:
 
$$\pi_{\opt}(G)\geq \min\left(2^k,\gamma_{k-1}(G)+2^{k-2}+1,\gamma_{k-2}(G)+1\right).$$
\label{uccso}
\end{thm}

\begin{proof}
The previous proof immediately gives the desired result if $|\supp(P)|\neq \gamma_{k-1}(G)$ or one of the inequalities in $2^{2-k}(|P|-|\supp(P)|)\geq W_{P\tau |N(\supp(P))}(v)\geq 1$ is strict. Therefore we investigate the case when $|\supp(P)|= \gamma_{k-1}(G)$ and show that one of the inequalities is strict. We use again the assumption that $|P|< \min(\gamma_{k-2}(G)+1,2^k)$ .

Suppose that $\gamma_{k-1}(G)=1$. Then $P$ contains pebbles only at a vertex $u$. Since $\supp(P)$ is still not a distance $k-2$ domination set, there is a vertex $v$ whose distance from $u$ is $k-1$. Thus $2^{k-1}$ pebbles at $u$ are required to reach $v$ and these are also enough. So $\pi_{\opt}(G)=2^{k-1}\geq \gamma_{k-1}(G)+2^{k-2}+1$.

Otherwise $\gamma_{k-1}(G)\geq 2$. 
Therefore for each $p\in\supp(P)$ there is a vertex $v$, 
such that the distance between $v$ and $p$ is $k-1$ but the distance between $v$ and $\supp(P)\setminus \{p\}$ is at least $k$. 

Fix $p$ and $v$ and choose a $\sigma$ sequence of pebbling moves which moves a pebble to $v$ and divide it to $\tau$ and $\mu$ like in the previous proof.

If $\tau$ removes more than two pebbles from a vertex, then at least two pebbles are consumed there and we have counted at most one consumption at each vertex, hence $|P|-|\supp(P)|>\left|P_{\tau}|_{N(\supp(P))}\right|$
 and the first inequality is strict. 

If $\tau$ contains a pebbling move which removes two pebbles from a $q\in \supp(P)\setminus \{p\}$,
 then this move places a pebble to a vertex $u$ whose distance from $v$ is $k-1$. If another move does not move forward this pebble,
 then $P_{\tau |N(\supp(P)}(u)>0$ and its coefficient in $W_{P\tau |N(\supp(P))}(v)$ is at most $2^{1-k}$ which is smaller than $2^{2-k}$ and the first equality is not possible. Else, a pebbling move removes
 two pebbles from $u$ and consumes a pebble. We have not counted this consumption, hence $|P|-|\supp(P)|>|P\tau |N(\supp(P))|$. 

The only remaining case is when $\tau$ contains only one pebbling move which moves a pebble from $p$ to a vertex $w$. $\mu$ can use only this pebble, but one pebble is not
 enough to apply a single pebbling move, therefore $\mu$ does nothing, $w$ is not $v$ because the distance between them is at least two, 
 so $\sigma$ does not move a pebble to $v$ which is a contradiction. Therefore this case is not possible.
\end{proof}

Using this last version of our result we can determine the optimal pebbling number of $K_3\square K_3\square K_5$.

\begin{cor}
The optimal pebbling number of $K_3\square K_3\square K_5$ is $6$.
\end{cor}

\begin{proof}
We have already seen a solvable distribution with size six in Figure \ref{ellenpelda}. It is not hard to see that the distance 2 domination number of $K_3\square K_3\square K_5$ is three: 

The support of the given distribution is a distance 2 domination set on three vertices. Two vertices are not enough. Consider a set $S$ whose size is two. The graph is vertex transitive, therefore it does not matter how we chose the first vertex $s_1$. In each copy of $K_3\square K_3$ which does not contain $s_1$ there are four undominated vertices whose distance from $s_1$ is more than two. The intersection of the closed neigbourhoods of the undominated vertices which are contained in the same $K_3\square K_3$ is empty. After we chose $s_2$ there will be a $K_3\square K_3$ which contains neither $s_1$ nor $s_2$. To reach its undominated vertices we have to move to a different $K_3\square K_3$, which consumes one of the two moves but our location in $K_3\square K_3$ does not change during this move. Only one more remained in $K_3\square K_3$, but this is not enough to arrive all four undominated vertices, because the intersection of their closed neighbourhoods is empty. Therefore $S$ is not a distance $2$ domination set. 
 
The domination number of $K_3\square K_3\square K_5$ is  more than 4: 

Consider a set of vertices $S$ whose size is $4$. The pigeonhole principle implies that there is a $K_3\square K_3$ which does not contain an element of $S$. Two vertex from the same $K_3\square K_3$ have some common adjacent vertices but all of them are contained in the same $K_3\square K_3$ where the two original vertices. Therefore each vertex in this $K_3\square K_3$ requires a different element of $S$ which dominates it. The order of $K_3\square K_3$ is nine, therefore $S$ is not a domination set.   

Finally we set $k$ to $3$ and apply Theorem \ref{uccso}. 
\end{proof}

\begin{figure}

\end{figure}

\section*{Acknowledgement}

The first and second authors are partially supported by the Hungarian National
Research Fund (grant number K116769).
The second and third authors are partially supported by the Hungarian National 
Research Fund (grant number K108947).

\end{document}